\documentclass{amsart}
\usepackage{amssymb}
\usepackage{amsmath}
\usepackage{amsfonts}
\usepackage{mathrsfs}

\newcommand{\mt}[1]{\mathtt{#1}}
\newcommand{\A}{\mathcal{A}}

\newcommand{\re}{\mathbb{R}}
\newcommand{\cpx}{\mathbb{C}}

\newcommand{\N}{\mathbb{N}}

\newcommand{\lmd}{\lambda}

\newcommand{\eps}{\epsilon}

\newcommand{\dt}{\delta}

\def\af{\alpha}
\def\bt{\beta}

\def\rank{\mbox{rank}}

\newcommand{\sig}{\sigma}
\newcommand{\Sig}{\Sigma}

\newcommand{\reff}[1]{(\ref{#1})}

\newcommand{\mc}[1]{\mathcal{#1}}

\newcommand{\bnum}{\begin{enumerate}}
\newcommand{\enum}{\end{enumerate}}

\newcommand{\bit}{\begin{itemize}}
\newcommand{\eit}{\end{itemize}}

\newcommand{\be}{\begin{equation}}
\newcommand{\ee}{\end{equation}}

\newcommand{\baray}{\begin{array}}
\newcommand{\earay}{\end{array}}

\newcommand{\bca}{\begin{cases}}
\newcommand{\eca}{\end{cases}}

\newcommand{\bcen}{\begin{center}}
\newcommand{\ecen}{\end{center}}

\newcommand{\bbm}{\begin{bmatrix}}
\newcommand{\ebm}{\end{bmatrix}}

\newcommand{\bpm}{\begin{pmatrix}}
\newcommand{\epm}{\end{pmatrix}}

\newcommand{\btab}{\begin{tabular}}
\newcommand{\etab}{\end{tabular}}

\newtheorem{theorem}{Theorem}[section]

\newtheorem{prop}[theorem]{Proposition}

\newtheorem{lemma}[theorem]{Lemma}

\newtheorem{defi}[theorem]{Definition}
\newtheorem{example}[theorem]{Example}

\newtheorem{algorithm}[theorem]{Algorithm}
\newtheorem{remark}[theorem]{Remark}

\def\beq#1{\begin{equation}\label{#1}}
\def\eeq{\end{equation}}
\def\bep{\begin{proof}}
\def\eps{\varepsilon}
\def\ep{\end{proof}}
\def\bt{\begin{theorem}}
\def\et{\end{theorem}}
\def\bl{\begin{lemma}}
\def\el{\end{lemma}}
\def\reff#1{(\ref{#1})}

\def\A{{\mathcal A}}
\def\B{\mathcal B}

\def\ignore#1{}

\def\qed{\hfill {$ \Box $} \medskip}
\def\bea#1{\begin{array}{#1}}
\def\ea{\end{array}}

\input epsf

\begin{document}

\title{\bf Real Eigenvalues of nonsymmetric tensors}

\author{Jiawang Nie}
\address{
Department of Mathematics,  University of California San Diego,  9500
Gilman Drive,  La Jolla,  California 92093,  USA.
} \email{njw@math.ucsd.edu}

\author{Xinzhen Zhang}
\address{
Department of Mathematics, School of Science, Tianjin University, Tianjin 300072, China.
} \email{xzzhang@tju.edu.cn}

\begin{abstract}
This paper discusses the computation of real $\mt{Z}$-eigenvalues
and $\mt{H}$-eigenvalues of nonsymmetric tensors.
A general nonsymmetric tensor has finitely many Z-eigenvalues,
while there may be infinitely many ones for special tensors.
In the contrast, every nonsymmetric tensor has finitely many $\mt{H}$-eigenvalues.
We propose Lasserre type semidefinite relaxation methods for computing such eigenvalues.
For every nonsymmetric tensor that has finitely many real $\mt{Z}$-eigenvalues,
we can compute all of them; each of them can be computed
by solving a finite sequence of semidefinite relaxations.
For every nonsymmetric tensor,
we can compute all its real $\mt{H}$-eigenvalues; each of them can be computed
by solving a finite sequence of semidefinite relaxations.
Various examples are demonstrated.
\end{abstract}



\maketitle

\section{Introduction}\label{Sec1}

For positive integers $m$ and $n_1,n_2,\cdots, n_m$, an $m$-order and
$(n_1,n_2,\cdots, n_m)$-dimensional real tensor is an array in the space
$\mathbb{R}^{n_1\times n_2\times\cdots\times n_m}$.
Every tensor $\A$ from this space can be indexed as
\be
\mathcal{A}=(\mathcal{A}_{i_1i_2\cdots i_m}), ~~1\leq i_j\leq n_j,~~j=1,2,\cdots m.
\ee
When $n_1=\cdots=n_m=n$, $\A$ is called an $m$-order $n$-dimensional tensor.
In such case, the tensor space $\mathbb{R}^{n_1\times n_2\times\cdots\times n_m}$
is denoted as $\mathtt{T}^m(\mathbb{R}^n)$.
A tensor in $\mathtt{T}^m(\mathbb{R}^n)$ is said to be {\it symmetric}
if its entries are invariant under permutations of indices $(i_1,i_2,\ldots, i_m)$.
The subspace of symmetric tensors in $\mathtt{T}^m(\mathbb{R}^n)$
is denoted as $\mathtt{S}^m(\mathbb{R}^n)$.
By replacing the real field $\re$ by the complex field $\cpx$,
the tensor spaces  $\mathtt{T}^m(\cpx^n)$ and $\mathtt{S}^m(\cpx^n)$
are similarly defined. Using the notation as in Qi~\cite{Q07},
for $\mathcal{A}\in\mathtt{T}^m(\cpx^n)$ and $x:=(x_1,\ldots, x_n)$, we denote
\be
\left\{\begin{array}{rl}
\mathcal{A}x^m & :=\sum\limits_{1\leq i_1,\cdots,i_m\leq n}
\mathcal{A}_{i_1i_2\cdots i_m}x_{i_1}x_{i_2}\cdots x_{i_m},\\
\mathcal{A}x^{m-1}& : = \Big( \sum\limits_{1\leq i_2,\cdots, i_m\leq n}
\mathcal{A}_{ji_2\cdots i_m}x_{i_2}\cdots x_{i_m} \Big)_{j=1,\ldots,n}.
\end{array}\right.
\ee
Note that $\mathcal{A} x^{m-1}$ is an $n$-dimensional vector.
Now we give some definitions of tensor eigenvalues
that are introduced in \cite{CDN14, CPZ08,L05,Q07}.

\begin{defi}
For $\mathcal{A}\in\mathtt{T}^m(\cpx^n)$, a number $\lambda\in \cpx$ is called a
$\mathtt{Z}$-eigenvalue of $\mathcal{A}$ if there exists a vector $u\in\cpx^n$ such that
\be \label{zdefinition}
\mathcal{A} u^{m-1}=\lambda u, \quad u^Tu =1.
\ee
(The superscript $^T$ denotes the transpose.)
Such $u$ is called a $\mathtt{Z}$-eigenvector associated with
$\lambda$, and such $(\lambda, u)$ is called a $\mathtt{Z}$-eigenpair.
\end{defi}

\begin{defi}\label{hdefinition}
For $\mathcal{A}\in\mathtt{T}^m(\cpx^n)$, a number $\lambda\in \cpx$
is called an $\mathtt{H}$-eigenvalue of $\mathcal{A}$
if there exists $0 \ne u\in  \cpx^n$ such that
\be \label{hdef}
\mathcal{A} u^{m-1}=\lambda u^{[m-1]}.
\ee
(The symbol $u^{[m-1]}$ denotes the vector such that $(u^{[m-1]})_i=(u_i)^{m-1}$
for $i=1,\ldots,n$).
Such $u$ is called an $\mathtt{H}$-eigenvector associated with $\lambda$,
and such $(\lambda, u)$ is called an $\mathtt{H}$-eigenpair.
\end{defi}

When $\lmd$ is a real $\mt{Z}$-eigenvalue (resp., $\mt{H}$-eigenvalue),
the associated $\mt{Z}$-eigenvector (resp., $\mt{H}$-eigenvector)
is not necessarily real. The $(\lambda, u)$ is called a real
$\mathtt{Z}$-eigenpair (resp., $\mathtt{H}$-eigenpair)
if both $\lmd$ and $u$ are real.
In the paper, we only discuss real eigenvalues. Throughout the paper, for convenience,
we call $\lmd$ a {\it real} $\mt{Z}$-eigenvalue (resp., $\mt{H}$-eigenvalue)
if {\it both $\lmd$ and $u$ are real}.

Tensor eigenvalues have broad applications in sciences and engineering.
They were introduced in Lim~\cite{L05} and Qi~\cite{Q05}.
The $\mt{Z}$-eigenvalus and $\mt{H}$-eigenvalues are useful
in signal processing, control, and diffusion imaging
(cf.~\cite{chen2013positive,ni2008eigenvalue,Qi03,Qi08,Qi10}).
For an introduction to the theory and applications of tensor computations,
we refer to \cite{Kolda2009, Lim13, QiSunWang07}.

When $\A$ is a real symmetric tensor, Cui et al. \cite{CDN14} discussed how
to compute all real $\mathtt{Z}$-eigenvalues and $\mathtt{H}$-eigenvalues.
There also exists work on computing partial eigenvalues, e.g., the biggest and smallest ones.
For the case $(m,n)=(3,2)$,  Qi et al. \cite{QWW09} discussed
how to compute  largest $\mathtt{Z}$-eigenvalues.
Shifted power methods are proposed for computing largest $\mathtt{Z}$-eigenvalues
(cf.~\cite{KM11,ZQY12}). In \cite{NW14}, a semidefinite relaxation method was proposed
to find best $\rank$-1 approximations, which can also be used for
computing  largest $\mathtt{Z}$-eigenvalues.
As shown in \cite{HiLi13,ZQY12}, it is NP-hard to compute
extreme eigenvalues of tensors.
For nonnegative tensors, the largest $\mathtt{H}$-eigenvalues
can be computed by methods based on
the Perron-Frobenius theorem (cf.~\cite{CPZ08,NQZ09}).

When $\A$ is a general nonsymmetric tensor, there is little work on
computing all real $\mathtt{Z}$-eigenvalues and $\mathtt{H}$-eigenvalues.
An elementary approach for this task is to solve the polynomial systems
\reff{zdefinition} and \reff{hdef} directly, for getting all complex solutions
by classical symbolic methods. This approach is typically quite
expensive and not practical, because of typically
high complexity of symbolic computations.

There are fundamental differences between symmetric and nonsymmetric
tensor eigenvalues.
It is known that every symmetric tensor has finitely many
$\mathtt{Z}$-eigenvalues (cf.~\cite{CDN14}).
However, this may not be true for nonsymmetric tensors.
A nonsymmetric tensor may have no real $\mathtt{Z}$-eigenvalues,
or may have infinitely many real ones.
We show such facts by the following examples.

\begin{example}\label{exampleofexistence}
Consider the tensor $\A \in\mathtt{T}^4(\mathbb{R}^2)$ such that
$\A_{ijkl}=0$ except
\[
\mathcal{A}_{1112}=\mathcal{A}_{1222}=1, \mathcal{A}_{2111}=\mathcal{A}_{2122}=-1.
\]
By the definition, $(\lambda, x)$ is a $\mathtt{Z}$-eigenpair if and only if
$$\left\{\begin{array}{rl}
&(x_1^2+x_2^2)x_2=\lambda x_1,\\
&-(x_1^2+x_2^2)x_1=\lambda x_2,\\
&x_1^2+x_2^2=1.
\end{array}\right.$$
One can check that the above does not have a real solution,
so $\mc{A}$ has no real $\mathtt{Z}$-eigenpairs.
By the definition, $(\lambda,x)$ is an $\mathtt{H}$-eigenpair if and only if
\[
\left\{\begin{array}{rl}
&(x_1^2+x_2^2)x_2=\lambda x_1^3,\\
&-(x_1^2+x_2^2)x_1=\lambda x_2^3,\\
& (x_1, x_2) \ne (0, 0).
\end{array}\right.
\]
One can similarly check that $\mc{A}$ has no real $\mathtt{H}$-eigenpairs.
\end{example}

\begin{example}\label{exampleofinfinity}
Consider the tensor $\mathcal{A}\in\mathtt{T}^4(\mathbb{R}^2)$
such that $\mc{A}_{ijkl}=0$ except
\[
\mathcal{A}_{1111}=\mathcal{A}_{2112}=1.
\]
Then, $(\lambda, x)$ is a $\mathtt{Z}$-eigenpair of $\mc{A}$  if and only if
\begin{equation}\left\{\begin{array}{cl}
x_1^3&=\lambda x_1,\\
x_1^2x_2&=\lambda x_2,\\
x_1^2+x_2^2&=1.
\end{array}
\right.
\end{equation}
One can check that every $\lambda \in [0,1]$ is a real $\mathtt{Z}$-eigenvalue,
with $4$ real $\mathtt{Z}$-eigenvectors $(\pm \sqrt{\lambda},\pm\sqrt{1-\lambda})$.
\end{example}

However, we would like to remark that the above examples are not general cases.
In fact, every generic nonsymmetric tensor has finitely many $\mt{Z}$-eigenvalues,
and its number can be given by explicit formula.
This is shown by Cartwright and Sturmfels \cite{CS13}.
Moreover, every nonsymmetric tensor has finitely many $\mt{H}$-eigenvalues.
By these facts, it is generally a well-posed question
to compute all real $\mt{Z}$-eigenvalues and $\mt{H}$-eigenvalues.

In this paper, we propose numerical methods for computing all
real $\mathtt{Z}$-eigenvalues (if there are finitely many ones)
and all $\mathtt{H}$-eigenvalues. For symmetric tensors, the
$\mathtt{Z}$-eigenvalues and $\mathtt{H}$-eigenvalues
are critical values of some polynomial optimization problems.
This property was significantly used in \cite{CDN14} for computing
all real eigenvalues. The method in \cite{CDN14}
is based on Jacobian SDP relaxations \cite{Nie-jac},
which are specially designed for solving polynomial optimization.
Indeed, the same kind of method can be used to
compute all local minima of polynomial optimization \cite{Ni13}.
However, the method in \cite{CDN14} are not suitable for computing
eigenvalues of nonsymmetric tensors, because their eigenvalues are no longer
critical values of polynomial optimization problems.

This paper is organized as follows. Section~2 gives some preliminaries
on polynomial optimization and tensor eigenvalues.
Section~3 proposes Lasserre type semidefinite relaxations for computing real
$\mt{Z}$-eigenvalues. If there are finitely many ones,
all the real $\mt{Z}$-eigenvalues can be found,
and each of them can be computed by solving a finite sequence of
semidefinite relaxations.
Section~4 proposes Lasserre type semidefinite relaxations for computing all real
$\mt{H}$-eigenvalues. Each of them can be computed by solving a finite sequence of
semidefinite relaxations. Numerical examples are shown in Section~5.

\section{Preliminaries}
\label{sec2:prelm}
\setcounter{equation}{0}

In this section, we review some basics in polynomial optimization.
We refer to \cite{LasBok,Lau} for surveys in the area.
In the space $\re^n$, the symbol $\| \cdot \|$ denotes the standard Euclidean norm.
Let $\mathbb{R}[x]$ be the ring of polynomials with real coefficients
and in variables $x:=(x_1, \ldots, x_n)$, and let $\mathbb{R}[x]_d$
be the set of real polynomials in $x$ whose degrees are at most $d$.
%
%
For a polynomial tuple $h=(h_1,h_2,\cdots, h_s)$,
the ideal generated by $h$ is the set
\[
I(h):= h_1\cdot \mathbb{R}[x]+h_2\cdot \mathbb{R}[x]+\cdots+h_s\cdot \mathbb{R}[x].
\]
The $k$-th truncation of $I(h)$ is the set
\[
I_k(h) :=  h_1\cdot \mathbb{R}[x]_{k-deg(h_1)}
+\cdots+h_s\cdot\mathbb{R}[x]_{k-deg(h_s)}.
\]
The complex and real algebraic varieties of $h$ are respectively defined as
\[
\mathcal{V}_{\cpx}(h):=\{x\in\cpx^n \, \mid \, h(x)=0 \},  \quad
\mathcal{V}_\mathbb{R}(h):=\mathcal{V}_{\cpx}(h) \cap \mathbb{R}^n.
\]
A polynomial $p$ is said to be sum of squares (SOS) if there exist
$p_1,p_2,\cdots p_r\in\mathbb{R}[x]$ such that $p=p_1^2+p_2^2+\cdots+p_r^2$.
The set of all SOS polynomials is denoted as $\Sig[x]$. For a given degree $m$, denote
\[
\Sig[x]_m :=  \Sig[x] \cap \mathbb{R}[x]_m .
\]
The quadratic module generated by a polynomial tupe $g=(g_1,\cdots, g_t)$ is the set

\[
Q(g) :=  \Sig[x]+ g_1 \cdot \Sig[x] +\cdots + g_t \cdot \Sig[x].
\]
The $k$-th truncation of the quadratic module $Q(g)$ is the set
\[
Q_k(g) := \Sig[x]_{2k} + g_1\cdot\Sig[x]_{2k-deg(g_1)}+
\cdots+ g_t \cdot\Sig[x]_{2k-deg(g_t)}.
\]
Note that if $g=\emptyset$ is an empty tuple, then
$Q(g) = \Sig[x]$ and $Q_k(g) = \Sig[x]_{2k}$.

Let $\N$ be the set of nonnegative integers.
For $x:=(x_1, \ldots,  x_n)$, $\af: = (\af_1,  \ldots,  \af_n)$
and a degree $d$, denote
\[
x^\af := x_1^{\af_1} \cdots x_n^{\af_n}, \quad
|\af| := \af_1 + \cdots + \af_n , \quad
\N_d^n := \{ \af \in \N^n: |\af| \leq d\}.
\]
Denote by $\re^{\N_d^n}$ the space of all real vectors $y$ that are
indexed by $\af \in \N_d^n$.
For $y \in \re^{\N_d^n}$, we can write it as
\[
y = (y_\af), \quad \af  \in \N_d^n.
\]
%
%
For $f = \sum_{ \af \in \N_d^n} f_\af x^\af \in \re[x]_d$ and $y \in \re^{\N_d^n}$,
we define the operation
\be \label{df:<fy>}
\langle f, y \rangle  \, := \,   \sum_{ \af \in \N_d^n} f_\af  y_\af .
\ee
For an integer $t \leq d$ and $y \in \re^{\N_d^n}$,  denote
the $t$-th truncation of $y$ as
\be \label{trunc:yt}
y|_{t} :=  (y_\af)_{ \af \in \N_t^n }.
\ee

Let $q \in \re[x]$ with $\deg(q) \leq 2k$. For each $y \in \re^{\N_{2k}^n}$,
$\langle q p^2, y \rangle$ is a quadratic form in $vec(p)$, the coefficient vector of
the polynomial $p$ with $\deg(qp^2) \leq 2k$.
Let $L_q^{(k)}(y)$ be the symmetric matrix such that
\be \label{df:Lqk(y)}
\langle q p^2, y \rangle =  vec(p)^T \Big( L_q^{(k)}(y) \Big) vec(p).
\ee
The matrix $L_q^{(k)}(y)$ is called the $k$-th localizing matrix of $q$
generated by $y$. It is linear in $y$.
For instance,  when $n=2$,  $k=2$ and $q=x_1x_2-x_1^2-x_2^2$,
\[
L_{x_1x_2-x_1^2-x_2^2}^{(2)}(y) =  \left(
\begin{array}{rcc}
y_{11}-y_{20}-y_{02} & y_{21}-y_{30}-y_{12} & y_{12}-y_{21}-y_{03} \\
y_{21}-y_{30}-y_{12} & y_{31}-y_{40}-y_{22} & y_{22}-y_{31}-y_{13} \\
y_{12}-y_{21}-y_{03} & y_{22}-y_{31}-y_{13} & y_{13}-y_{22}-y_{04} \\
\end{array}
\right).
\]
If $q= (q_1, \ldots, q_r)$ is a tuple of polynomials, we then define
\[
L_q^{(k)}(y) := \Big( L_{q_1}^{(k)}(y), \ldots, L_{q_r}^{(k)}(y) \Big).
\]
When $q=1$ (the constant $1$ polynomial),
$L_1^{(k)}(y)$ is called the $k$-th moment matrix generated by $y$, and we denote
\be \label{df:Mk(y)}
M_k(y) :=  L_1^{(k)}(y).
\ee
For instance,  when $n=2$ and $k=2$,
 \[
 M_2(y)= \left(
\begin{array}{cccccc}
      y_{00} & y_{10} & y_{01} & y_{20} & y_{11} & y_{02} \\
      y_{10} & y_{20} & y_{11} & y_{30} & y_{21} & y_{12} \\
      y_{01} & y_{11} & y_{02} & y_{21} & y_{12} & y_{03} \\
      y_{20} & y_{30} & y_{21} & y_{40} & y_{31} & y_{22} \\
      y_{11} & y_{21} & y_{12} & y_{31} & y_{22} & y_{13} \\
      y_{02} & y_{12} & y_{03} & y_{22} & y_{13} & y_{04} \\
\end{array}
  \right).
\]
For a degree $d$,  denote the monomial vector
\be \label{vec:[x]d}
[x]_d  := \bbm 1 \,  x_1 \,  \cdots \,  x_n \,  x_1^2 \,  x_1x_2
\,  \cdots \,  x_n^2 \,  \cdots \,  x_1^m \, \cdots \, x_n^m \ebm^T.
\ee

As shown in Example \ref{exampleofinfinity}, there may be infinitely many
$\mathtt{Z}$-eigenvalues. But this is not the case for a general nonsymmetric tensor.
In \reff{zdefinition}, a real $\mt{Z}$-eigenpair $(\lmd, u)$
of a tensor $\A$ is called {\it isolated}
if there exists $\eps>0$ such that no other real $\mt{Z}$-eigenpair $(\mu, v)$
satisfies $| \lmd - \mu | + \| u - v \| < \eps$.
Similarly, a real $\mt{Z}$-eigenvalue $\lmd$ is called {\it isolated}
if there exists $\eps>0$ such that no other real $\mt{Z}$-eigenvalue $\mu$
satisfies $| \lmd - \mu | < \eps$. In practice, we need to check whether
a $\mt{Z}$-eigenvalue or $\mt{Z}$-eigenpair is isolated or not. Denote
\[
F(\lmd, x) := \bbm x^Tx - 1 \\ \mc{A} x^{m-1} - \lmd x  \ebm.
\]
Then, $(\lambda, u)$ is a $\mathtt{Z}$-eigenpair of $\A$ if and only if $F(\lambda, u)=0$.
Let $J(\lmd, x)$ be the Jacobian matrix of
the vector function $F(\lmd, x)$ with respect to $(\lmd, x)$.

\begin{lemma}\label{isolated}
Let $\A \in \mt{T}^m(\re^n)$ and $\lmd$ be a real $\mathtt{Z}$-eigenvalue of $\A$.
\bit

\item [(i)] If $u$ is a real $\mathtt{Z}$-eigenvector associated with $\lmd$ and
$J(\lambda, u)$ is nonsingular, then $(\lmd, u)$ is isolated.

\item [(ii)] If $u_1, \ldots, u_N$ are the all real $\mathtt{Z}$-eigenvectors of $\A$
associated to $\lmd$ and each $J(\lambda, u_i)$ is nonsingular,
then $\lambda $ is an isolated $\mathtt{Z}$-eigenvalue of $\A$.

\eit

\end{lemma}
\begin{proof}
(i) We prove it by a contradiction argument. Suppose otherwise the $\mathtt{Z}$-eigenpair
$(\lambda, u)$ is not isolated. Then there exists a sequence
$ \{ (\lambda_l, u^{(l)}) \}_{l=1}^{\infty}$  of $\mathtt{Z}$-eigenpairs such that each
$(\lambda_l, u^{(l)}) \neq (\lmd , u)$ and $(\lambda_l, u^{(l)}) \to (\lambda, u)$. Let
\[
d^{(l)} = \bbm  \lmd_l - \lmd \\  u^{(l)} - u \ebm.
\]
By the second order Taylor expansion, we have
\[
0=F(\lambda_l,  u^{(l)}) = F(\lmd,u)+J(\lmd, u) d^{(l)}
+O( \| d^{(l)} \|^2 ).
\]
Note that $F(\lmd,u)=0$ and each $\| d^{(l)} \| \ne 0$. The above implies that
\[
J(\lmd, u) \Big( d^{(l)} / \| d^{(l)} \| \Big) = O( \| d^{(l)} \| ).
\]
Since each $d^{(l)} / \| d^{(l)} \| $ has unit length,
we can generally assume that $d^{(l)} / \| d^{(l)} \|  \to \hat{d}$.
Then $\|\hat{d} \|=1$ and $\hat{d} \ne 0$. The above implies that
\[
J(\lmd, u) \hat{d} = 0,
\]
contradicting the nonsingularity of $J(\lmd,u)$.
So, $(\lmd,u)$ is isolated.

(ii) Suppose otherwise that $\lmd$ is not isolated. Then there exists a sequence
$\{ \lmd_l \}$ of distinct real $\mt{Z}$-eigenvalues such that
$\lmd_l \to \lmd$. Each $\lmd_l$ has a real $\mt{Z}$-eigenvector $u^{(l)}$.
Since  $u^{(l)}$ has unit length, we can generally assume $u^{(l)} \to \hat{u}$.
Thus, $(\lmd_l, u^{(l)}) \to (\lmd, \hat{u})$. Clearly, $(\lmd, \hat{u})$
is also a real $\mt{Z}$-eigenpair, but it is not isolated.
By the assumption, $\hat{u}$ is one of $u_1, \ldots, u_N$.
So, one of $(\lmd, u_i)$ is not isolated, which is a contradiction.
\end{proof}

\section{Computing $\mathtt{Z}$-eigenvalues}
\label{sec3:Zeig}
\setcounter{equation}{0}

Let $\A \in \mt{T}^m(\re^n)$ be a tensor.
Recall that $(\lmd, u)$ is a $\mt{Z}$-eigenpair of $\A$ if
$\A u^{m-1} = \lmd u$ and $u^T u = 1$. So,
\[
\lmd = \lmd u^T u =  u^T \A u^{m-1} =  \A u^m .
\]
Hence, $u$ is a $\mt{Z}$-eigenvector if and only if
\[
\A u^{m-1} = (\A u^m) u, \quad u^T u = 1,
\]
and the associated $\mt{Z}$-eigenvalue is $\A u^m$.
A general nonsymmetric tensor in $\mt{T}^m(\re^n)$ has finitely many
$\mt{Z}$-eigenvalues, as shown in \cite{CS13}.
For special tensors, there might be infinitely many ones
(cf.~Example \ref{exampleofinfinity}).

In this section, we aim at computing all real $\mathtt{Z}$-eigenvalues
when there are finitely many ones.
Let $h$ the polynomial tuple:
\be \label{df:h:Zeig}
h = ( \A x^{m-1} - (\A x^m)x, \, x^Tx-1 ).
\ee
Then, $u$ is a $\mathtt{Z}$-eigenvector of $\mathcal{A}$ if and only if $h(u)=0$.
Let $Z(\re,\A)$ denote the set of real $\mathtt{Z}$-eigenvalues of $\A$.
If it is a finite set, we list $Z(\re,\A)$ monotonically as
\[
\lambda_1< \lambda_2< \cdots < \lmd_N.
\]
We aim at computing them sequentially, from the smallest to the biggest.

\subsection{The smallest $\mathtt{Z}$-eigenvalue}

To compute the smallest $\mathtt{Z}$-eigenvalue $\lmd_1$,
we consider the polynomial optimization problem
\be \label{smallesteigen}
\min \quad f(x):=\mathcal{A} x^m \, \quad \,
 {\mbox s.t.}  \quad h(x) = 0,
\ee
where $h$ is as in \reff{df:h:Zeig}.
Note that $u$ is a $\mathtt{Z}$-eigenvector if and only if $h(u)=0$,
with the $\mt{Z}$-eigenvalue $f(u)$.
The optimal value of (\ref{smallesteigen}) is $\lmd_1$, if it exists. Let
\be \label{df:k0:Zeg}
k_0 =  \lceil (m+1)/2 \rceil .
\ee
Lasserre's hierarchy \cite{L01} of semidefinite relaxations
for solving \reff{smallesteigen} is
\be  \label{relaxsdp}
\left\{  \begin{array}{rl}
 f_1^{1,k} := \min  & \langle f,y\rangle\\
 {\mbox s.t.}& \langle 1,y\rangle = 1, \, L^{(k)}_{h}(y)=0,   \\
 & M_k(y)\succeq 0, y \in \re^{ \N_{2k}^n },
 \end{array} \right.
\ee
for the orders $k=k_0, k_0+1,\ldots$.
See \reff{df:Lqk(y)}-\reff{df:Mk(y)} for the notation
$L^{(k)}_{h}(y)$ and $M_k(y)$. In the above,
$X\succeq 0$ means that the matrix $X$ is positive semidefinite.
The dual optimization problem of (\ref{relaxsdp}) is
\be  \label{relaxsos}
\left\{ \begin{array}{rl}
 f_1^{2,k} :=\max& \gamma\\
 {\mbox s.t.}& f-\gamma\in I_{2k}(h)+\Sig[x]_{2k}.
 \end{array}  \right.
\ee
As in \cite{L01}, it can be shown that for all $k$
\[
 f_1^{2,k}  \leq   f_1^{1,k}  \leq \lambda_1
\]
and the sequences $\{  f_1^{1,k}  \}$ and $\{  f_1^{2,k}  \}$
are monotonically increasing.

\begin{theorem}\label{Zthm:cvg:lmd1}
Let $\mathcal{A}\in\mathtt{T}^m(\re^n)$ and
$Z(\re, \A)$ be the set of its real $\mt{Z}$-eigenvalues.
Then we have:

\bit

\item [(i)] The set $Z(\re, \A) = \emptyset$ if and only if
the semidefinite relaxation \reff{relaxsdp} is infeasible
for some order $k$.

\item [(ii)] If $Z(\re, \A) \ne \emptyset$ and
$\lmd_1$ is the smallest $\mt{Z}$-eigenvalue, then
\be \label{lmd1:asym:cvg}
\lim_{k \to \infty} f_1^{2,k} =   \lim_{k \to \infty} f_1^{1,k}  = \lambda_1.
\ee
If, in addition, $Z(\re, \A)$ is a finite set, then for all $k$ sufficiently big
\be \label{lmd1:finite:cvg}
f_1^{2,k} =     f_1^{1,k}  = \lambda_1.
\ee

\item [(iii)] Suppose $y^*$ is a minimizer of (\ref{relaxsdp}).
If there exists $t\leq k$ such that
\be \label{flat:Mt(y*)}
{\mbox rank} \, M_{t-k_0}(y^*)={\mbox rank}\, M_t(y^*),
\ee
then $f_1^{1,k}  = \lambda_1$ and there are $r:=\rank M_t(y^*)$
distinct real $\mt{Z}$-eigenvectors $u_1,\ldots,u_r$
associated with $\lmd_1$.

\item [(iv)] Suppose $Z(\re, \A)$ is a finite set.
If there are finitely many real $\mt{Z}$-eigenvectors
associated with $\lmd_1$, then, for all $k$ big enough and
for every minimizer $y^*$ of (\ref{relaxsdp}),
\reff{flat:Mt(y*)} is satisfied for some $t\leq k$.

\eit

\end{theorem}

\begin{proof}
%
%
(i) ``if" direction: this is obvious. If $\A$ has a real $\mt{Z}$-eigenpair
$(\lmd,u)$, then $[u]_{2k}$ (see \reff{vec:[x]d} for the notation)
is feasible for \reff{relaxsdp}, a contradiction.

``only if" direction: If $\A$ has no real $\mt{Z}$-eigenvalues,
then the equation $h(x)=0$ has no real solutions.
By Positivstellensatz (cf.~\cite{BCR}), $-1 \in I(h)  + \Sig[x]$.
So, when $k$ is big enough, $-1 \in I_{2k}(h)  + \Sig[x]_{2k}$,
and then \reff{relaxsos} is unbounded from above.
By weak duality, \reff{relaxsdp} must be infeasible,
for all $k$ big enough.

(ii)  Note that $x^Tx-1$ is a polynomial in the tuple $h$.
So, $-(x^Tx-1)^2 \in I(h)$ and the set $-(x^Tx-1)^2 \geq 0$ is compact.
The ideal $I(h)$ is archmedean (cf.~\cite{L01}).  The asymptotic convergence
\reff{lmd1:asym:cvg} can be implied by Theorem~4.2 of \cite{L01}.

Next, we prove the finite convergence \reff{lmd1:finite:cvg}
when $Z(\re, \A) \ne \emptyset$ is a finite set.
Write $Z(\re, \A) = \{ \lmd_1, \ldots, \lmd_N \}$, with
$\lmd_1 < \cdots < \lmd_N$.
Let $b_1, \ldots, b_N \in \re[t]$
be the univariate real polynomials in $t$ such that
$b_i(\lmd_j) = 0$ when $i \ne j$ and $b_i(\lmd_j) = 1$ when $i = j$.
For $i=1, \ldots, N$, let
\[
s_i :=  ( \lmd_i -  \lmd_1) \Big( b_i (  f(x) )  \Big)^2.
\]
Let $s := s_1+ \cdots + s_N$. Then,
$s \in \Sig[x]_{2k_1}$ for some $k_1 > 0$.
The polynomial
\[
\hat{f}:=f - \lmd_1 - s
\]
vanishes identically on $\mc{V}_{\re}(h)$.
By Real Nullstellensatz (cf.~\cite[Corollary~4.1.8]{BCR}),
there exist an integer $\ell>0$ and $q \in \Sig[x]$ such that
\[
\hat{f}^{2\ell} + q \in  I(h).
\]
For all $\eps >0$ and $c>0$, we can write
\[
\hat{f} + \eps = \phi_\eps + \theta_\eps \, \quad
\mbox{ where}
\]
\[
\phi_\eps = -c \eps^{1-2\ell} (\hat{f}^{2\ell}+q), \quad
\theta_\eps = \eps \Big(1 + \hat{f}/\eps + c ( \hat{f}/\eps)^{2\ell} \Big)
+ c \eps^{1-2\eps} q.
\]
By Lemma~2.1 of \cite{Ni13siam}, when $c \geq \frac{1}{2\ell}$,
there exists $k_2$ such that, for all $\eps >0$,
\[
\phi_\eps \in I_{2k_2}(h), \quad \theta_\eps \in \Sig[x]_{2k_2}.
\]
Hence, we can get
\[
f - (\lmd_1 -\eps) = \phi_\eps + \sig_\eps,
\]
where $\sig_\eps = \theta_\eps + s\in \Sig[x]_{2k_2}$ for all $\eps >0$.
This implies that, for all $\eps>0$,
$\gamma = \lmd_1-\eps$ is feasible in \reff{relaxsos} for the order $k_2$.
Thus, we get $f_1^{2,k_2} \geq \lmd_1$.
Note that $f_1^{2,k} \leq f_1^{1,k} \leq \lmd_1$ for all $k$ and
the sequence $\{  f_1^{2,k}  \}$ is monotonically increasing.
So, \reff{lmd1:finite:cvg} must be true when $k \geq k_2$.

(iii) Note that $M_t(y^*) \succeq 0$ and $L_h^{(t)}(y^*) = 0$, because $t\leq k$.
When \reff{flat:Mt(y*)} is satisfied, by Theorem~1.1
of \cite{CF05}, there exist $r:= \rank \, M_t(y^*)$ vectors
$u_1,\ldots, u_r \in \mc{V}_{\re}(h)$ such that
\[
y^*|_{2t} = c_1 [u_1]_{2t} + \cdots + c_r [u_r]_{2t},
\]
with numbers $c_1, \ldots, c_r > 0$. The condition $\langle 1, y^* \rangle =1$ implies that
\[
c_1 + \cdots + c_r = 1.
\]
By the notation $\langle \cdot, \cdot \rangle $ as in \reff{df:<fy>},
we can see that $\langle f, [u_i]_{2k} \rangle = f(u_i)$, so
\[
f_1^{1,k} = \langle f, y^* \rangle =
c_1 f(u_1) + \cdots + c_r f(u_r),
\]
\[
f_1^{1,k} \leq f(u_i) \quad i = 1, \ldots, r,
\]
because each $[u_i]_{2k}$ is feasible for \reff{relaxsdp}. Thus,
\[
f_1^{1,k} =  f(u_1) =  \cdots = f(u_r).
\]
Also note that $f_1^{1,k} \leq \lmd_1$ and each $u_i$ is a
$\mt{Z}$-eigenvector. So,
\[
f_1^{1,k} =  f(u_1) =  \cdots = f(u_r) = \lmd_1.
\]

(iv) When $Z(\re, \A)$ is a finite set, the sequences $\{ f_1^{1,k} \}$
and $\{ f_1^{2,k} \}$ have finite convergence to $\lmd_1$, by item (iii).
If \reff{smallesteigen} has finitely many minimizers,
i.e., $\lmd_1$ has finitely many real $\mt{Z}$-eigenvectors,
the rank condition \reff{flat:Mt(y*)} must be satisfied
when $k$ is sufficiently big. This can be implied by
Theorem~2.6 of \cite{Ni13mp}.
\end{proof}

\begin{remark}
\label{rmk:Zeig:lmd1}  \rm
(1) The rank condition \reff{flat:Mt(y*)} can be used as a criterion
to check whether $f_1^{1,k}  = \lambda_1$ or not.
If it is satisfied,  then we can get $r$
distinct minimizers $u_1,  \ldots,  u_r$ of \reff{smallesteigen},
i.e., each $u_i$ is $\mt{Z}$-eigenvector of $\A$.
The vectors $u_i$ can be computed by the method in Henrion and Lasserre~\cite{HL05}. \\
(2) Suppose \reff{flat:Mt(y*)} holds.
If $\rank\,  M_k(y^*)$ is maximum among the set of all optimizers of \reff{relaxsdp},
then we can get all mimizers of \reff{smallesteigen} (cf.~\cite[\S6.6]{Lau}), i.e.,
we can get all real $\mt{Z}$-eigenvectors associated to $\lmd_1$.
When \reff{relaxsdp}-\reff{relaxsos} are solved by
primal-dual interior point methods, typically we can get
all $\mt{Z}$-eigenvectors associated to $\lmd_1$,
if there are finitely many ones.
\end{remark}

\subsection{Bigger $\mathtt{Z}$-eigenvalues}\label{ithzeigenvalue}

Suppose the $i$th smallest real $\mt{Z}$-eigenvalue $\lmd_i$ is isolated and is known.
We want to determine whether the next bigger one $\lmd_{i+1}$ exists or not.
If it exists, we show how to compute it.
Let $\dt >0$ be a small number.
Consider the polynomial optimization problem
\be  \label{itheigenpair}
\left\{ \begin{array}{rl}
\min & f(x)\\
{\mbox s.t.}& h(x) = 0, \, f(x)\geq \lambda_{i} + \dt.
\end{array} \right.
\ee
Clearly, the optimal value of \reff{itheigenpair} is the smallest $\mt{Z}$-eigenvalue
that is greater than or equal to $\lmd_i + \dt$.  Let
\be \label{omg:lmdi+dt}
\omega(\lmd_i + \dt) := \min \{ \lmd \in Z(\re, \A) : \, \lmd \geq \lmd_i+\dt \}.
\ee
Lasserre's hierarchy of semidefinite relaxations for
solving \reff{itheigenpair} is
\be \label{ithrelaxsdp}
\left\{ \begin{array}{rl}
f_{i+1}^{1,k} := \min& \langle f,y\rangle\\
 {\mbox s.t.}& \langle 1, y \rangle = 1, \,  L^{(k)}_{h}(y)=0, \\
 & M_k(y)\succeq 0, \, L_{f-\lambda_{i}-\dt}^{(k)}(y)\succeq 0,
 y \in \re^{ \N_{2k}^n },
\end{array} \right.
\ee
for the orders $k = k_0, k_0+1, \ldots$.
The dual problem of \reff{ithrelaxsdp} is
\be \label{ithrelaxsos}
\left\{ \begin{array}{rl}
 f_{i+1}^{2,k} :=\max  & \gamma\\
 {\mbox s.t.}&  f-\gamma  \in I_{2k}(h)+Q_k(f-\lambda_{i}-\dt).
 \end{array} \right.
\ee

Similar to Theorem \ref{Zthm:cvg:lmd1}, we have the following convergence result.

\begin{theorem}\label{Zcvg:lmd:i+1}
Let $\mathcal{A}\in\mathtt{T}^m(\re^n)$ and $Z(\re, \A)$
be the set of real $\mt{Z}$-eigenvalues of $\A$.
Suppose $\lmd_i \in Z(\re,\A)$. Then we have:

\bit

\item [(i)] The intersection $Z(\re, \A) \cap [\lmd_i+\dt, +\infty) = \emptyset$ if and only if
the semidefinite relaxation \reff{ithrelaxsdp} is infeasible
for some order $k$.

\item [(ii)] If $Z(\re, \A) \cap [\lmd_i+\dt, +\infty) \ne \emptyset$, then
\be \label{lmd:i+1:cvg:asmp}
\lim_{k \to \infty} f_{i+1}^{2,k} =   \lim_{k \to \infty} f_{i+1}^{1,k}  = \omega(\lmd_i+\dt).
\ee
If, in addition, $Z(\re, \A) \cap [\lmd_i+\dt, +\infty)$ is a finite set, then
\be \label{lmd:i+1:fntcvg}
f_{i+1}^{2,k}  =   f_{i+1}^{1,k}  = \omega(\lmd_i+\dt)
\ee
for all $k$ sufficiently big.

\item [(iii)] Suppose $y^*$ is a minimizer of (\ref{ithrelaxsdp}).
If \reff{flat:Mt(y*)} is satisfied for some $t\leq k$, then
$f_{i+1}^{1,k}  = \omega(\lmd_i+\dt)$ and
there are $r:=\rank M_t(y^*)$ distinct real $\mt{Z}$-eigenvectors $u_1,\ldots,u_r$,
associated with the $\mt{Z}$-eigenvalue $\omega(\lmd_i+\dt)$.

\item [(iv)]  Suppose $Z(\re, \A) \cap [\lmd_i+\dt, +\infty)$
is a finite set and $\omega(\lmd_i+\dt)$ has finitely many $\mt{Z}$-eigenvectors.
Then, for all $k$ big enough, and for every minimizer $y^*$ of (\ref{ithrelaxsdp}),
there exists $t\leq k$ satisfying  \reff{flat:Mt(y*)}.

\eit

\end{theorem}

\begin{proof}
%
%
The proof is mostly same as for Theorem~\ref{Zthm:cvg:lmd1}.
In the following, we only list the differences.

(i) For the ``only if" direction: If $Z(\re,\A) \cap [\lmd_i+\dt, +\infty) = \emptyset$,
then \reff{itheigenpair} is infeasible.
By Positivstellensatz, $-1 \in I(h)  + Q(f-\lmd_i-\dt)$.
The resting proof is same as for Theorem~\ref{Zthm:cvg:lmd1}(i).

(ii) The ideal $I(h)$ is archmedean, and so is $I(h)+ Q(f-\lmd_i-\dt)$.
The asymptotic convergence
\reff{lmd:i+1:cvg:asmp} can be implied by Theorem~4.2 of \cite{L01}.
To prove the finite convergence \reff{lmd:i+1:fntcvg},
we follow the same proof as for Theorem~\ref{Zthm:cvg:lmd1}(ii).
Suppose $Z(\re, \A) \cap [\lmd_i+\dt, +\infty) = \{ \nu_1, \ldots, \nu_L \}$.
 Construct the polynomial $s$ same as there and let
\[
\hat{f} = f - \omega(\lmd_i+\dt) - s.
\]
Then $\hat{f}$ vanishes identically on the set
$\{x \in \re^n: \, h(x)=0, f(x) - \lmd_i - \dt \geq 0\}$.
By Real Nullstellensatz (cf.~\cite[Corollary~4.1.8]{BCR}),
there exist an integer $\ell>0$ and $q \in  Q(f - \lmd_i - \dt)$
such that $\hat{f}^{2\ell} + q \in  I(h)$.
The resting proof is same, except replacing $\Sig[x]$ by $Q(f - \lmd_i - \dt)$,
and $\Sig[x]_{2k}$ by $Q_k(f - \lmd_i - \dt)$.

(iii)-(iv)  The proof is same as for Theorem~\ref{Zthm:cvg:lmd1}(iii)-(iv).
\end{proof}

The convergence of semidefinite relaxations of \reff{ithrelaxsdp}-\reff{ithrelaxsos}
can be checked by the condition \reff{flat:Mt(y*)}.
When it is satisfied, the $\mt{Z}$-eigenvectors $u_1,\ldots,u_r$
can be computed by the method in \cite{HL05}.
Typically, we can get all $\mt{Z}$-eigenvectors
if primal-dual interior-point methods are used to solve the semidefinite programs.
We refer to Remark~\ref{rmk:Zeig:lmd1}.

Next, we show how to use $\omega(\lmd_i+\dt)$ to determine $\lmd_{i+1}$.
Assume that $\lmd_i$ is isolated, otherwise there are infinitely many
$\mt{Z}$-eigenvalues and it is impossible to get all of them.
If $\lmd_i$ is the biggest $\mt{Z}$-eigenvalue, then we stop;
otherwise, the next bigger one $\lmd_{i+1}$ exists.
For such case, if $\dt>0$ in \reff{itheigenpair} is small enough,
then $\omega(\lmd_i+\dt)$ as in \reff{omg:lmdi+dt}
equals $\lmd_{i+1}$. Consider the optimization problem
\be  \label{opt:chi:k}
\left\{ \baray{rl}
\nu_i := \max & f(x) \\
  s.t. &  h(x)=0, \,  f(x) \leq \lambda_i + \delta.
\earay \right.
\ee
The optimal value of \reff{opt:chi:k} is the biggest $\mt{Z}$-eigenvalue
of $\A$ that is smaller than or equal to $\lmd_i+\dt$, i.e.,
\[
\nu_i = \max\{\lmd \in \mt{Z}(\re,\A) : \, \lmd \leq \lmd_i + \dt \}.
\]
The next bigger $\mt{Z}$-eigenvalue $\lmd_{i+1}$ can be
determined by the following theorem.

\begin{theorem}  \label{thm:cmp:lmdi+1}
Let $\mathcal{A}\in\mathtt{T}^m(\re^n)$ and $\dt>0$.
Suppose $\lmd_i$ is a $\mt{Z}$-eigenvalue of $\A$
and $\lmd_{max}$ is the biggest one.
\bit

\item [(i)] Suppose $\lmd_i$ is an isolated $\mt{Z}$-eigenvalue of $\A$.
If $\nu_i = \lmd_i$ and $\lmd_{max} > \lmd_i$, then
$\omega(\lmd_i+\dt)=\lmd_{i+1}$, which is
the smallest $\mt{Z}$-eigenvalue bigger than $\lmd_i$.

\item [(ii)] If $\nu_i = \lmd_i$ and \reff{ithrelaxsdp} is infeasible for some $k$,
then $\lmd_i = \lmd_{max}$.

\eit
\end{theorem}
\begin{proof}
(i) We have seen that $\nu_i = \lmd_i $ is the biggest $\mt{Z}$-eigenvalue
of $\A$ that is smaller than or equal to $\lmd_i+\dt$.
Since $\lmd_i = \nu_i$ is isolated and $\lmd_{max} > \lmd_i$,
the smallest $\mt{Z}$-eigenvalue bigger than $\lmd_i$ is
$\lmd_{i+1}$, and $\lmd_{i+1} > \lmd_i + \dt$.
By \reff{omg:lmdi+dt}, $\omega(\lmd_i+\dt)$ is the smallest
$\mt{Z}$-eigenvalue that is greater than or equal to $\lmd_i+\dt$.
There are no $\mt{Z}$-eigenvalues in the open interval $(\lmd_i, \lmd_{i+1})$.
So, $\omega(\lmd_i+\dt) = \lmd_i+\dt$.

(ii) When \reff{ithrelaxsdp} is infeasible for some $k$,
by Theorem~\ref{Zcvg:lmd:i+1}(i), all the real
$\mt{Z}$-eigenvalues are smaller than $\lmd_i + \dt$.
Note that $\nu_i$ is the biggest $\mt{Z}$-eigenvalue that
is smaller than or equal to $\lmd_i+\dt$.
If $\lmd_i = \nu_i$, then $\lmd_i$ must be the
biggest $\mt{Z}$-eigenvalue, i.e.,
$\lmd_i = \lmd_{max}$.
\end{proof}

The problem \reff{opt:chi:k} is a polynomial optimization.
Its optimal value $\nu_i$ can also be computed by solving Lasserre type
semidefinite relaxations that are similar to \reff{ithrelaxsdp}-\reff{ithrelaxsos}.
When $\lmd_i$ is isolated, for $\dt>0$ sufficiently small, we must have
$\nu_i = \lmd_i$, no matter if $\lmd_{i+1}$ exists or not.
This is because $\nu_i$ is the smallest $\mt{Z}$-eigenvalue
greater than or equal to $\lmd_i+\dt$.
Lemma~\ref{isolated} can be used to verify that $\lmd_i$ is isolated.

\subsection{An algorithm for computing all real $\mathtt{Z}$-eigenvalues}
\label{algozeig}

Assume all the real $\mt{Z}$-eigenvalues of the tensor $\A$ are isolated.
We compute all of them, from the smallest to the biggest.
First,  we compute $\lmd_1$ if it exists,
by solving \reff{relaxsdp}-\reff{relaxsos}.
After getting $\lmd_1$,  we solve the hierarchy of
\reff{ithrelaxsdp}-\reff{ithrelaxsos} and then determine $\lmd_2$.
If $\lmd_2$ does not exist, we stop;
otherwise, we then determine $\lmd_3$.
Repeating this procedure, we can get all the real $\mt{Z}$-eigenvalues.

The following algorithm can be applied to get $\mt{Z}$-eigenvalues.

\begin{algorithm}\label{algzeig}
Compute real $\mathtt{Z}$-eigenvalues of a tensor $\A \in \mt{T}^m(\re^n)$.

\bit

\item [Step~0:] Choose a small positive value for $\dt$ (e.g., $0.05$).

\item [Step~1:]  If (\ref{relaxsdp}) is infeasible for some order $k$, then $\A$ has no
real $\mathtt{Z}$-eigenvalues and stop.
Otherwise, solve it to get the smallest $\mathtt{Z}$-eigenvalue $\lambda_1$.
Let $i=1$.

\item [Step~2:]   Solve \reff{opt:chi:k} for $\nu_i$.
If $\nu_i=\lambda_{i}$, go to Step~3.
Otherwise, reduce the value of $\dt$
(e.g., let $\dt := \dt/5$) and compute $\nu_i$.
Repeat until we get $\nu_i=\lambda_{i}$.

\item [Step~3:] If (\ref{ithrelaxsdp}) is infeasible for some order $k$,
the largest $\mathtt{Z}$-eigenvalue is
$\lambda_{i}$ and stop. Otherwise, solve it for $\omega(\dt_i+\dt)$.
Let $\lmd_{i+1} := \omega(\dt_i+\dt)$, $i:=i+1$ and go to Step 2.

\eit
\end{algorithm}

For a generic tensor $\A$, it has finitely many $\mt{Z}$-eigenvalues,
and all of them are isolated. So,
Algorithm~\ref{algzeig} terminates after a finite number of steps,
for almost all nonsymmetric tensors in $\mt{T}^m(\re^n)$.

\section{Computing all H-eigenvalues}
\label{sec4:Heig}
\setcounter{equation}{0}

In this section, we compute all $\mathtt{H}$-eigenvalues for nonsymmetric tensors.
Unlike $\mt{Z}$-eigenvalues, the number of $\mathtt{H}$-eigenvalues
is always finite.

\begin{prop} \label{number:Heig}
Every tensor $\A \in \mt{T}^m(\cpx^n)$ has $n(m-1)^{n-1}$ complex $\mathtt{H}$-eigenvalues,
including their multiplicities.
\end{prop}
\begin{proof}
Let $\mathcal{I}\in\mathtt{T}^m(\mathbb{R}^n)$ be the identity tensor
whose only non-zero entries are
$\mathcal{I}_{ii\cdots i}=1$, with $i=1,2,\cdots, n$.
Recall that $\lmd$ is an $\mathtt{H}$-eigenvalue if and only if
there exists $0 \ne u \in \cpx^n$ such that
$\mathcal{A}u^{m-1}=\lambda u^{[m-1]}$,
that is,  $(\mathcal{A}-\lambda \mathcal{I}) u^{m-1}=0$.
By the definition of resultant (cf.~\cite{Stu02}), which we denote by $Res$,
$\lmd$ is an $\mathtt{H}$-eigenvalue if and only if
\be \label{Res:A-lmdI=0}
Res\big( (\mathcal{A}- \lambda \mathcal{I}) x^{m-1} \big) = 0.
\ee
The resultant $Res\big( (\mathcal{A}- \lambda \mathcal{I}) x^{m-1} \big)$
is homogeneous in the entries of $\A$ and $\lmd$.
It has degree $D := n(m-1)^{n-1}$. We can expand it as
\[
Res\big( (\A- \lambda \mathcal{I}) x^{m-1} \big)=p_0(\mathcal{A})  + p_1(\A) \lmd
+ \cdots + p_D(\A) \lmd^D.
\]
By the homogeneity of $Res$,
$p_D(\A) = Res\big( - \mathcal{I} x^{m-1} \big) \ne 0$,
because the homogeneous polynomial system $ -\mathcal{I} x^{m-1} =0$
has no nonzero complex solutions. This means that
$Res\big( (\mathcal{A}- \lambda \mathcal{I}) x^{m-1} \big)$
is not constantly zero, and its degree is $D$.
Therefore, \reff{Res:A-lmdI=0} has $D$ complex roots including multiplicities.
Hence, $\A$ has $D$ complex $\mt{H}$-eigenvalues.
\end{proof}

Recall that $(\lmd, u)$ is a real $\mt{H}$-eigenpair
of a tensor $\A \in \mt{T}^m(\re^n)$ if and only if
$0 \ne u \in \re^n$ and $ \A u^{m-1} = \lmd u^{[m-1]}$.
Let $m_0$ be the biggest even number not bigger than $m$, i.e.,
\[
m_0 = 2 \lceil (m-1)/2 \rceil.
\]
Note that $m-1 \leq m_0 \leq m$. We can normalize such $u$ as
\be \label{u^m==1}
(u_1)^{m_0} + \cdots + (u_n)^{m_0} = 1.
\ee
Under this normalization, the $\mt{H}$-eigenvalue $\lmd$ can be given as
\[
\lmd = \lmd (u^{[m_0-m+1]})^T u^{[m-1]} =  (u^{[m_0-m+1]})^T \A u^{m-1}.
\]
Let $h$ be the polynomial tuple
\be \label{df:h:Heig}
h := \Big( \A x^{m-1} - \big( (x^{[m_0-m+1]})^T \A x^{m-1} \big) x^{[m-1]},
\sum_{i=1}^n(x_i)^{m_0} - 1 \Big).
\ee
Then, $u$ is an $\mt{H}$-eigenvector normalized as in \reff{u^m==1}
if and only if $h(u) = 0$.
Since $\A$ has finitely many $\mt{H}$-eigenvalues,
we can order its real ones monotonically as
\[
\mu_1 < \mu_2 < \cdots < \mu_N,
\]
if at least one real $\mt{H}$-eigenvalue exists.
We call $\mu_i$ the $i$th smallest $\mt{H}$-eigenvalue.

\subsection{The smallest $\mathtt{H}$-eigenvalue}

In this subsection, we show how to determine $\mu_1$.
Let $h$ be the polynomial tuple as in \reff{df:h:Heig},
then $\mu_1$ equals the optimal value of the optimization problem
\begin{equation}
\label{opt:H:mu1}
\left\{ \begin{array}{rl}
\min & f(x):= (x^{[m_0-m+1]})^T \A x^{m-1}  \\
{\mbox s.t.}& h(x) = 0.
\end{array} \right.
\end{equation}
Lasserre's hierarchy \cite{L01} of semidefinite relaxations for solving
\reff{opt:H:mu1} is
\begin{equation}
 \label{Heig:mom:mu1}
\left\{ \begin{array}{rl}
  \rho_1^{1,k} :=  \min& \langle  f,z  \rangle\\
 {\mbox s.t.}&  \langle 1, z \rangle = 1, \, L_{h}^{(k)} (z)= 0 , \\
 &  M_k(z)\succeq 0, \, z \in \re^{ \N_{2k}^n },
 \end{array} \right.
 \end{equation}
for the orders $k=k_0, k_0+1, \ldots, $ where
\be \label{Heig:k0}
k_0 := \lceil (m_0+m-1)/2 \rceil.
\ee
The dual optimization problem of \reff{Heig:mom:mu1} is
 \begin{equation} \label{Heig:sos:mu1}
\left\{ \begin{array}{rl}
  \rho_1^{2,k} :=  \max& \gamma \\
 {\mbox s.t.}& f - \gamma \in  I_{2k}(h) +\Sig[x]_{2k}.
  \end{array} \right.
\end{equation}
As can be shown in \cite{L01},
$\rho_1^{2,k}  \leq  \rho_1^{1,k} \leq \mu_1$
for all $k$, and the sequences $\{ \rho_1^{1,k} \}$
and $\{ \rho_1^{2,k} \}$ are monotonically increasing.

\begin{theorem}\label{thm:Heig:mu1}
Let $\mathcal{A}\in\mathtt{T}^m(\re^n)$
and $H(\re, \A)$ be the set of its real $\mt{H}$-eigenvalues.
Then we have:

\bit

\item [(i)] The set $H(\re, \A) = \emptyset$ if and only if
the semidefinite relaxation \reff{Heig:mom:mu1} is infeasible
for some order $k$.

\item [(ii)] If $H(\re, \A) \ne \emptyset$,
then for all $k$ sufficiently large
\be \label{mu1:ficvg:Heig}
\rho_1^{1,k}   =  \rho_1^{2,k}  =   \mu_1.
\ee

\item [(iii)] Let $k_0$ be as in \reff{Heig:k0}.
Suppose $z^*$ is a minimizer of \reff{Heig:mom:mu1}.
If there exists an integer $t\leq k$  such that
\be \label{flat:Mt(z*)}
{\mbox rank} \, M_{t-k_0}(z^*)={\mbox rank}\, M_t(z^*),
\ee
then $\rho_1^{1,k}  = \mu_1$ and
there are $r:=\rank M_t(z^*)$ distinct real $\mt{H}$-eigenvectors $u_1,\ldots,u_r$
associated with $\mu_1$ and normalized as in \reff{u^m==1}.

\item [(iv)] Suppose $\A$ has finitely many $\mt{H}$-eigenvectors associated to $\mu_1$.
Then, for all $k$ big enough and for every minimizer $z^*$ of (\ref{Heig:mom:mu1}),
there exists an integer $t\leq k$ satisfying \reff{flat:Mt(z*)}.

\eit

\end{theorem}
\begin{proof}
It can be proved in the same way as for Theorem~\ref{Zthm:cvg:lmd1}.
The only difference is that $H(\re, \A)$ is always a finite set,
by Proposition~\ref{number:Heig}. For cleanness of the paper,
we omit the proof here.
\end{proof}

The rank condition \reff{flat:Mt(z*)} is a criterion
for checking the convergence of Lasserre's hierarchy
of \reff{Heig:mom:mu1} and \reff{Heig:sos:mu1}.
When it is satisfied, the
$\mt{H}$-eigenvectors $u_1,\ldots,u_r$
can be computed by the method in \cite{HL05}.
Typically, we can get all $\mt{H}$-eigenvectors
if primal-dual interior-point methods are used to solve the semidefinite relaxations.
We refer to Remark~\ref{rmk:Zeig:lmd1}.

 \subsection{Bigger $\mathtt{H}$-eigenvalues}\label{ithheigenvalue}

Suppose the $i$th smallest real $\mt{H}$-eigenvalue $\mu_i$ exists and is known.
We want to determine whether the next bigger one $\mu_{i+1}$ exists or not.
If it exists, we show how to compute it.

Let $\dt >0$ be a small number. Consider the optimization problem
\be \label{minf:Hmu:i+1}
\left\{ \begin{array}{rl}
\min & f(x)\\
{\mbox s.t.}& h(x) = 0, \, f(x)\geq \mu_i+\delta,
\end{array} \right.
\ee
where $f,h$ are same as in \reff{opt:H:mu1}.
The optimal value of \reff{minf:Hmu:i+1} is the smallest
$\mt{H}$-eigenvalue of $\A$ that is greater than or equal to $\mu_i+\dt$. Denote
\be \label{Hvpi:>mui+dt}
\varpi(\mu_i+ \dt)  := \min \{ \mu \in H(\re, \A): \, \mu \geq \mu_i + \dt\}.
\ee
Lasserre's hierarchy of semidefinite relaxations for solving \reff{minf:Hmu:i+1} is
\be \label{egH:mom:mui+1}
\left\{ \begin{array}{rl}
 \rho_{i+1}^{1,k} :=  \min& \langle f,z\rangle\\
 {\mbox s.t.}& \langle 1, z \rangle = 1,  L_h^{(k)}(z) = 0, \\
 & M_k(z)\succeq 0, \, L_{f-\mu_i-\dt}^{(k)}(z)\succeq 0, \, z \in \re^{ \N_{2k}^n },
 \end{array} \right.
\ee
for the orders $k = k_0, k_0+1, \ldots$.
The dual problem of \reff{egH:mom:mui+1} is
\be \label{egH:sos:mui+1}
\left\{ \begin{array}{rl}
 \rho_{i+1}^{2,k} := \max  & \gamma \\
  {\mbox s.t.}&  f -\gamma  \in  I_{2k}(h) + +Q_k(f-\mu_i-\dt).
  \end{array} \right.
\ee
The properties of relaxations \reff{egH:mom:mui+1}-\reff{egH:sos:mui+1}
are as follows.

\begin{theorem}\label{Hcvg:mu:i+1}
Let $\mathcal{A}\in\mathtt{T}^m(\re^n)$ and $H(\re, \A)$
be the set of its real $\mt{H}$-eigenvalues.
Assume that $\mu_i \in H(\re, \A)$. Then we have:

\bit

\item [(i)] The intersection $H(\re, \A) \cap [\mu_i+\dt, +\infty) = \emptyset$ if and only if
the semidefinite relaxation \reff{egH:mom:mui+1} is infeasible
for some order $k$.

\item [(ii)] If $H(\re, \A) \cap [\mu_i+\dt, +\infty) \ne \emptyset$,
then for all $k$ sufficiently large
\be \label{Hmu:i+1:fntcg}
\rho_{i+1}^{2,k}   =     \rho_{i+1}^{1,k}  = \varpi(\mu_i+\dt).
\ee

\item [(iii)] Let $z^*$ be a minimizer of (\ref{egH:mom:mui+1}).
If \reff{flat:Mt(z*)} is satisfied for some $t\leq k$, then
there exists $r:=\rank M_t(z^*)$ $\mt{H}$-eigenvectors $u_1,\ldots,u_r$
that are normalized as in \reff{u^m==1} and
are associated with $\varpi(\mu_i+\dt)$.

\item [(iv)]  Suppose $\A$ has finitely many $\mt{H}$-eigenvectors that are
associated with $\varpi(\mu_i+\dt)$ and are normalized as in \reff{u^m==1}.
Then, for all $k$ big enough and for all minimizer $z^*$ of (\ref{egH:mom:mui+1}),
there exists $t\leq k$ satisfying \reff{flat:Mt(z*)}.

\eit

\end{theorem}
\begin{proof}
It can proved in the same way as for Theorem~\ref{Zcvg:lmd:i+1}.
Note that $H(\re,\A)$ is always a finite set, by Proposition~\ref{number:Heig}.
\end{proof}

In the following, we show how to use $\varpi(\mu_i+\dt)$
to determine $\mu_{i+1}$. Consider the maximization problem
\be  \label{max:H:nui}
\left\{ \baray{rl}
\upsilon_i := \max & f(x) \\
  s.t. &  h(x)=0, \,  f(x) \leq \mu_i + \delta.
\earay \right.
\ee
The optimal value of \reff{max:H:nui} is the biggest $\mt{H}$-eigenvalue
of $\A$ that is smaller than or equal to $\mu_i+\dt$, i.e.,
\[
\upsilon_i = \max\{\mu \in H(\re,\A) : \, \mu \leq \mu_i + \dt \}.
\]
The next bigger $\mt{H}$-eigenvalue $\mu_{i+1}$ can be
determined by the following theorem.

\begin{theorem}  \label{decide:mui+1}
Let $\mathcal{A}\in\mathtt{T}^m(\re^n)$ and $\dt>0$.
Suppose $\mu_i \in H(\re, \A)$
and $\mu_{max}$ is the maximum $\mt{H}$-eigenvalue.
Let $\varpi(\lmd_i+\dt)$ be as in \reff{Hvpi:>mui+dt}.
\bit

\item [(i)] If $\upsilon_i = \mu_i$ and $\mu_{max} > \mu_i$, then
$\varpi(\lmd_i+\dt) = \mu_{i+1}$.

\item [(ii)] If $\upsilon_i = \mu_i$ and \reff{egH:mom:mui+1}
is infeasible for some $k$, then $\mu_i = \mu_{max}$.

\eit
\end{theorem}
\begin{proof}
The proof is same as for Theorem~\ref{thm:cmp:lmdi+1}.
Note that $\A$ has finitely many $\mt{H}$-eigenvalues,
and $\mu_i$ is always an isolated one.
\end{proof}

Since \reff{max:H:nui} is a polynomial optimization,
the optimal value $\upsilon_i$ can also be computed by solving Lasserre type
semidefinite relaxations that are similar to \reff{egH:mom:mui+1}-\reff{egH:sos:mui+1}.
For $\dt>0$ sufficiently small, we must have
$\upsilon_i = \mu_i$, no matter if $\mu_{i+1}$ exists or not.
This is because $\upsilon_i$ is the biggest $\mt{H}$-eigenvalue
that is less than or equal to $\mu_i+\dt$.

\subsection{An algorithm for all $\mathtt{H}$-eigenvalues}

We can compute all real $\mt{H}$-eigenvalues of a tensor $\A$ sequentially,
from the smallest one to the biggest one, if they exist.
A similar version of Algorithm~\ref{algzeig} can be applied.

\begin{algorithm} \label{alg:Heig}
Compute all real $\mathtt{H}$-eigenvalues of a tenosr $\A \in \mt{T}^m(\re^n)$.
\bit

\item [Step~0:] Choose a small positive value for $\dt$
(e.g., $0.05$).

\item [Step~1:] If (\ref{Heig:mom:mu1}) is infeasible for some order
$k$, then $\A$ has no real $\mathtt{H}$-eigenvalues and stop.
Otherwise, compute the smallest one $\mu_1$
by solving the hierarchy of (\ref{Heig:mom:mu1}). Let $i:=1$.

\item [Step~2:] Solve (\ref{max:H:nui}) for its optimal value $\upsilon_i$.
If $\upsilon_i=\mu_i$, go to Step~3; otherwise,
reduce $\dt$ (e.g., $\dt = \dt/5$) and compute $\upsilon_i$.
Repeat until we get $\upsilon_i = \mu_i$.

\item [Step~3:] If (\ref{egH:mom:mui+1}) is infeasible
for some order $k$, the largest $\mathtt{H}$-eigenvalue is
$\mu_i$ and stop. Otherwise, compute $\varpi(\mu_i+\dt)$
 by solving the hierarchy of (\ref{egH:mom:mui+1}). Let
$\mu_{i+1} :=\varpi(\mu_i+\dt)$, $i:=i+1$ and go to Step~2.

\eit
\end{algorithm}

\section{Numerical Examples}
\label{sc:num}

In this section, we give numerical examples
for how to compute real $\mathtt{Z}$-eigenvalues and $\mathtt{H}$-eigenvalues
for nonsymmetric tensors. The computation is implemented in MATLAB~7.10
in a Dell Linux Desktop with 8GB memory and Intel(R) CPU 2.8GHz.
The software Gloptipoly 3 \cite{HLL} is used to solve the semidefinite relaxations.
For computational results, only four decimal digits are displayed,
for cleanness of the presentation.

For odd ordered tensors, the $\mathtt{Z}$-eigenvalues always appear in $\pm $ pairs,
so only nonnegative $\mathtt{Z}$-eigenvalues are shown for them.
For Algorithm~\ref{algzeig} to compute all of them,
the real $\mathtt{Z}$-eigenvalues of
the tensor need to be all isolated. General nonsymmetric tensors
have finitely many $\mathtt{Z}$-eigenvalues, so they are all isolated.
In applications, for particular tensors,
Lemma~\ref{isolated} can be used to verify that all real $\mathtt{Z}$-eigenvalues
are isolated, once they are computed. In all our examples,
we used it to check the isolatedness.
The computations of all real eigenvalues in our examples
took from a few to a couple of seconds.

\begin{example}(\cite[Example~3]{LMV00})\label{example311}
Consider the tensor $\mathcal{A}\in\mathtt{T}^4(\mathbb{R}^2)$ with entries
$\mc{A}_{i_1i_2i_3i_4}=0$ except
\[
\mathcal{A}_{1111}=25.1,~~\mathcal{A}_{1212}=25.6,
~~\mathcal{A}_{2121}=24.8,~~\mathcal{A}_{2222}=23.
\]
Applying Algorithms~\ref{algzeig} and \ref{alg:Heig}, we get
all the real $\mathtt{Z}$-eigenvalues and $\mathtt{H}$-eigenvalues correctly.
The computed eigenvalues, as well as their eigenvectors,
are shown in Table~\ref{tabexample311}.
\begin{table}
\caption {$\mathtt{Z}$/$\mt{H}$-eigenpairs of the tensor in Example 5.1} \label{tabexample311}
\begin{center}
\begin{tabular}{|r|c|c|c|} \hline $i$ & $1$ & $2$ & $3$ \\\hline
$\mathtt{Z}$-eigenvalue $\lmd_i$ & $23.000$ & $25.1000$ &  \\\hline
$\mathtt{Z}$-eigenvectors &$\pm(0, 1)$ & $\pm(1, 0)$&\\\hline
$\mathtt{H}$-eigenvalue $\mu_i$ & $23.0000$ & $25.1000$ & $49.2687$ \\\hline
$\mathtt{H}$-eigenvector & $\pm(0,1)$ & $\pm (1,0)$ & $\pm(0.8527, \pm 0.8285)$\\\hline
\end{tabular}
\end{center}
\end{table}

\end{example}

\begin{example}(\cite[Example~1]{Qi11}) \label{exa2}
Consider the tensor $\mathcal{A}\in\mathtt{T}^3(\mathbb{R}^3)$ with
the entries $\mc{A}_{i_1i_2i_3}=0$ except {\small
\[
\begin{array}{lcccr}
\mathcal{A}_{111}=0.4333,&\mathcal{A}_{121}=0.4278,&\mathcal{A}_{131}=0.4140,&\mathcal{A}_{211}=0.8154,&\mathcal{A}_{221}=0.0199,\\
\mathcal{A}_{231}=0.5598,&\mathcal{A}_{311}=0.0643,&\mathcal{A}_{321}=0.3815,&\mathcal{A}_{331}=0.8834,&\mathcal{A}_{112}=0.4866,\\
\mathcal{A}_{122}=0.8087,&\mathcal{A}_{132}=0.2073,&\mathcal{A}_{212}=0.7641,&\mathcal{A}_{222}=0.9924,&\mathcal{A}_{232}=0.8752,\\
\mathcal{A}_{312}=0.6708,&\mathcal{A}_{322}=0.8296,&\mathcal{A}_{332}=0.125,&\mathcal{A}_{113}=0.3871,&\mathcal{A}_{123}=0.0769,\\
\mathcal{A}_{133}=0.3151,&\mathcal{A}_{213}=0.1355,&\mathcal{A}_{223}=0.7727,&\mathcal{A}_{233}=0.4089,&\mathcal{A}_{313}=0.9715,\\
\mathcal{A}_{323}=0.7726,&\mathcal{A}_{333}=0.5526.
\end{array}
\] \noindent}By
Algorithms~\ref{algzeig} and \ref{alg:Heig},
we get all its nonnegative $\mt{Z}$-eigenvalues
\[
0.2331, \quad 0.4869, \quad 2.7418,
\]
and all its real $\mt{H}$-eigenvalues
\[
1.3586, \quad 1.4985, \quad 1.5226, \quad 4.7303.
\]
\end{example}

\begin{example}(\cite[\S4.1]{IAD13})\label{exa3}
Consider the tensor $\mathcal{A}\in\mathtt{T}^3(\mathbb{R}^3)$
with $\mc{A}_{i_1i_2i_3}=0$ except {\smaller \smaller
\[
\begin{array}{lcccr}
\mathcal{A}_{111}=0.0072,&\mathcal{A}_{121}=-0.4413,&\mathcal{A}_{131}=0.1941,&\mathcal{A}_{211}=-0.4413,&\mathcal{A}_{221}=0.0940,\\
\mathcal{A}_{231}=0.5901,&\mathcal{A}_{311}=0.1941,&\mathcal{A}_{321}=-0.4099,&\mathcal{A}_{331}=-0.1012,&\mathcal{A}_{112}=-0.4413,\\
\mathcal{A}_{122}=0.0940,&\mathcal{A}_{132}=-0.4099,&\mathcal{A}_{212}=0.0940,&\mathcal{A}_{222}=0.2183,&\mathcal{A}_{232}=0.2950,\\
\mathcal{A}_{312}=0.5901,&\mathcal{A}_{322}=0.2950,&\mathcal{A}_{332}=0.2229,&\mathcal{A}_{113}=0.1941,&\mathcal{A}_{123}=0.5901,\\
\mathcal{A}_{133}=-0.1012,&\mathcal{A}_{213}=-0.4099,&\mathcal{A}_{223}=0.2950,&\mathcal{A}_{233}=0.2229,&\mathcal{A}_{313}=-0.1012,\\
\mathcal{A}_{323}=0.2229,&\mathcal{A}_{333}=-0.4891.
\end{array}
\] \noindent}By
Algorithms~\ref{algzeig} and \ref{alg:Heig},
we get all its nonnegative $\mathtt{Z}$-eigenvalues
\[
0.0000, \quad 0.5774
\]
and all its real $\mathtt{H}$-eigenvalues
\[
0.0000, \quad 0.7875.
\]
\end{example}

\begin{example}(\cite[Example~3.19]{NW14})\label{exa319}
Consider the tensor $\mathcal{A}\in\mathtt{T}^3(\mathbb{R}^n)$  such that
$$\mathcal{A}_{i_1i_2i_3}=tan\left(i_1-\frac{i_2}{2}+\frac{i_3}{3}\right).$$
Applying Algorithms~\ref{algzeig} and \ref{alg:Heig},
we get all the real $\mathtt{Z}$-eigenvalues and $\mathtt{H}$-eigenvalues.
They are reported in Table~\ref{tableofexa319}, for $n=2,3,4$.
There are no real $\mathtt{H}$-eigenvalues for the case $n=2$.

\begin{table}
\caption {$\mathtt{Z}$/$\mt{H}$-eigenvalues of the tensor in Example \ref{exa319}}
\label{tableofexa319}
\begin{center}
\begin{tabular}{|l|l|l|l|} \hline $n$ & $\mathtt{Z}$-eigenvalues $(\geq 0)$
& $\mathtt{H}$-eigenvalues
\\\hline $n=2$& $10.5518$   & none
\\\hline
$n=3$&  $0.2336$,\, $1.6614$,\,  $10.5063$  &  $-2.5615$, \, $0.3456$  \\  \hline
$n=4$&  $3.3651$,\, $8.8507$,\,  $10.4981$
& $-6.2888$, \, $-0.7048$,\,  $2.8947$,\, $5.9245$ \\  \hline
\end{tabular}
\end{center}
\end{table}
\end{example}

\begin{example}\label{atanexample}
Consider the tensor $\mathcal{A}\in \mathtt{T}^4(\mathbb{R}^3)$ such that
\[
\mathcal{A}_{i_1\cdots i_4}=arctan(i_1 i_2^2 i_3^3 i_4^4).
\]
By Algorithms~\ref{algzeig} and \ref{alg:Heig},
we get all its real $\mathtt{Z}$-eigenvalues
\[
-0.2700, \quad 0.0003, \quad 13.8286,
\]
and all its $\mathtt{H}$-eigenvalues
\[
-0.3662, \quad 0.0005, \quad 41.4705.
\]
\end{example}

\begin{example}\label{mulexample}
Consider the tensor $\mathcal{A}\in \mathtt{T}^4(\mathbb{R}^3)$ such that
\[
\mathcal{A}_{i_1\cdots i_4}=(1+ i_1 + 2i_2 + 3i_3 + 4i_4  )^{-1}.
\]
By Algorithms~\ref{algzeig} and \ref{alg:Heig}, we get all its
real $\mathtt{Z}$-eigenvalues
\[
0.0000, \quad 0.0002, \quad 0.4572,
\]
and all its real $\mathtt{H}$-eigenvalues
\[
0.0000, \quad 0.0005, \quad 1.3581.
\]
\end{example}

\begin{example}\label{expexample}
Consider the tensor $\mathcal{A}\in\mathtt{T}^5(\mathtt{R}^n)$ such that
\[
\mathcal{A}_{i_1,\cdots, i_5}=\Big(
\sum\limits_{j=1}^5 (-1)^{j-1} \exp(i_j) \Big)^{-1}.
\]
For $n=2,3,4,5$, all the real Z-eigenvalues and H-eigenvalues
are found by Algorithms~\ref{algzeig} and \ref{alg:Heig}.
They are shown in Table \ref{tabexpexample}.
\begin{table}
\caption {$\mathtt{Z}$/$\mathtt{H}$-eigenvalues of
the tensor in Example \ref{expexample}} \label{tabexpexample}
\begin{center}
\begin{tabular}{|l|l|l|l|} \hline
$n$ &$\mathtt{Z}$-eigenvalues$(\geq 0)$&$\mathtt{H}$-eigenvalues  \\ \hline
$2$ &  $0.4721$ & $0.5138,\,1.2654$ \\  \hline
$3$ & $0.6158$ & $0.5196,\,2.0800,\,2.2995,\, 2.4335$ \\ \hline
$4$ & $0.7682$ & $0.5199,\,2.0964,\,2.2980,\,2.3991,\,2.9454,\,4.4609,$ \\
&&  $4.9588,\,5.4419$ \\ \hline
$5$ & $0.8384$ & $0.5199,\,2.0978,\,2.2997,\,2.3860,\,2.4010,\,2.9658,$ \\
&& $\,4.4713,\,4.4902,\,4.6880,\,4.7008,\,5.0136,\,5.7891,$ \\
&& $6.0668,\,7.3250,\,7.3469,\,8.8555$ \\ \hline
\end{tabular}
\end{center}
\end{table}
\end{example}

\begin{example}\label{sqrexample}
Consider the tensor $\mathcal{A}\in \mathtt{T}^3(\mathbb{R}^n)$ such that
\[
\mathcal{A}_{i_1 i_2 i_3} = \frac{1}{10}
\left( i_1 + 2i_2 + 3i_3 - \sqrt{ i_1^2 + 2i_2^2 + 3i_3^2 }\right).
\]
For the values $n=2,3,4$,
we get all the real $\mathtt{Z}$-eigenvalues and $\mathtt{H}$-eigenvalues,
by Algorithms~\ref{algzeig} and \ref{alg:Heig}.
They are reported in Table~\ref{tabsqrexample}.
\begin{table}
\caption {$\mathtt{Z}$/$\mathtt{H}$-eigenvalues of the tensor in Example \ref{sqrexample}}
\label{tabsqrexample}
\begin{center}
\begin{tabular}{|l|l|l|} \hline
$n$ & $\mathtt{Z}$-eigenvalues$(\geq 0)$ &  $\mathtt{H}$-eigenvalues  \\  \hline
$2$ & $ 0.0024,\,  0.0038,\,  1.4928$ & $0.0060,\,   2.0960$  \\ \hline
$3$ & $ 0.0067,\,   0.0161,\,  3.6417$ &
$-0.0401,\, -0.0243,\, 0.0086,\, 0.0235,\, 0.1568,\,$\\&&
$ 0.6635,\, 1.4958,\,  6.2378$ \\ \hline
$4$ & $0.0000,\, 0.0107,\,  0.0396,\,  6.9922$ &
$-0.0240,\,\ -0.0087,\, 0.0000,\, 0.0102,\, 0.0258,\,$\\ &&$0.0437,\,  1.5761, \,\
2.6824,\,  4.1089,\, 5.8270,\, 13.7960$ \\  \hline
\end{tabular}
\end{center}
\end{table}

\end{example}

\noindent
{\bf Acknowledgement}
Jiawang Nie was partially supported by the NSF grants
DMS-0844775 and DMS-1417985. Xinzhen Zhang was partially supported by the National Natural Science Foundation of China (Grant No. 11471242, 11301303 and 11171180).

\end{document}